\documentclass{amsart}
\usepackage[utf8]{inputenc}
\usepackage{amsthm}
\usepackage{amsmath}
\usepackage{amssymb}
\usepackage{mathrsfs}
\usepackage{enumitem, hyperref}
\usepackage{changepage}
\usepackage{graphics}
\usepackage{pifont}
\usepackage{adjustbox}

\usepackage{tikz}
\usepackage{tikz-cd}

\usepackage{float}

\newtheorem{thm}{Theorem}[section]
\newtheorem{lem}[thm]{Lemma}

\newtheorem{prop}[thm]{Proposition}
\newtheorem{cor}[thm]{Corollary}

\theoremstyle{definition}
\newtheorem{defi}[thm]{Definition}

\newtheorem{ex}[thm]{Example}

\theoremstyle{remark}
\newtheorem{remark}[thm]{Remark}
\newtheorem{question}[thm]{Question}

\newtheorem{warning}[thm]{Warning}

\DeclareMathOperator{\Hom}{\mathcal{H}om}

\DeclareMathOperator{\End}{End}

\newcommand{\Addresses}{{
  \bigskip
  \footnotesize
  C.~Hone, \textsc{Department of Mathematical Sciences, University of Copenhagen, Denmark}\par\nopagebreak
  \textit{E-mail address}: \texttt{cth@math.ku.dk}
}}

\title{Real algebraic varieties and their intersection cohomology}
\author{Chris Hone}

\begin{document}
\maketitle
\author
\begin{abstract}
For a real variety with smooth point, we construct a complex of sheaves on its real points which behaves like intersection cohomology. This real geometric extension is (shifted) Verdier self-dual, and occurs as a direct summand within any resolution of singularities. The stalks of this object provide lower bounds for the Betti numbers of the real fibres in any resolution of singularities. On the real flag variety, the category of real geometric extensions along the Schubert stratification is equivalent to the corresponding category of even parity sheaves on the complex flag variety.
\end{abstract}
\section{Introduction}

A smooth complete $d$-dimensional algebraic variety $X$ over the real numbers has an associated manifold $X(\mathbb{R})$ of real points, and if $X(\mathbb{R})$ is nonempty, one has a fundamental class in the top degree homology $H_d(X(\mathbb{R}),\mathbb{F}_2)$. The cap product with this class induces the Poincaré duality isomorphism:
\[H^\ast(X(\mathbb{R}),\mathbb{F}_2)\to 
H_{d-\ast}(X(\mathbb{R}),\mathbb{F}_2).\]
If $X$ is singular, one still has a fundamental class in top degree homology, but this map is not necessarily an isomorphism. Over the complex numbers, where any singular $Y$ has a fundamental class over any field $k$, this map has a canonical factorisation through the (middle perversity) intersection cohomology groups of Goresky-MacPherson: \[\begin{tikzcd}
H^\ast(Y(\mathbb{C}),k)\arrow[r,"\cap \text{[}Y(\mathbb{C})\text{]} "]\arrow[d]&H_{2d-\ast}(Y(\mathbb{C}),k)\\
IH^\ast(Y(\mathbb{C}),k)\arrow[r,"\sim"]&IH^{2d-\ast}(Y(\mathbb{C}),k)^\vee\arrow[u]
\end{tikzcd}\]
These groups $IH^\ast(Y(\mathbb{C}),k)$ are naturally self-dual via a perfect intersection pairing, and are a fundamental invariant of the topological space $Y(\mathbb{C})$. An old question of Goresky-MacPherson \cite[Q.7]{Goresky1984} asks if a similar family of intersection cohomology groups exists for $Y(\mathbb{R})$ with mod two coefficients. Our main result is the construction of such groups.

\begin{thm}\label{Thm:coh of real geo ext}
Let $Y$ be a complete $d$-dimensional irreducible real algebraic variety admitting a smooth real point. Then there exists a family of $\mathbb{F}_2$ vector spaces $\mathscr{E}^i(Y(\mathbb{R}),\mathbb{F}_2)$ with the following properties:

\begin{itemize}
    \item There exists\footnote{This is induced by a (noncanonical) choice of duality on the derived level, see Proposition \ref{propertiesrealgeoext}.} a Poincaré duality isomorphism: \[\mathscr{E}^i(Y(\mathbb{R}),\mathbb{F}_2)\to \mathscr{E}^{d-i}(Y(\mathbb{R}),\mathbb{F}_2)^\vee.\]
     \item These vector spaces are finite dimensional, and vanish outside the dimension of $Y(\mathbb{R})$:\[\mathscr{E}^i(Y(\mathbb{R}),\mathbb{F}_2)=0 \text{ if } i>d \text{ or } i<0.\]
    \item One has canonical comparison maps \[H^i(Y(\mathbb{R}),\mathbb{F}_2)\xrightarrow{\gamma_i} \mathscr{E}^i(Y(\mathbb{R}),\mathbb{F}_2)\] which factorise the cap product with the fundamental class: \[\begin{tikzcd}
H^i(Y(\mathbb{R}),\mathbb{F}_2)\arrow[r,"\cap \text{[}Y(\mathbb{R})\text{]} "]\arrow[d,"\gamma_i"]&H_{d-i}(Y(\mathbb{R}),\mathbb{F}_2)\\
\mathscr{E}^i(Y(\mathbb{R}),\mathbb{F}_2)\arrow[r,"\sim"]&\mathscr{E}^{d-i}(Y(\mathbb{R}),\mathbb{F}_2)^\vee\arrow[u,"\gamma_{d-i}^\vee"']
\end{tikzcd}\]
    \item These maps $\gamma_i$ are all isomorphisms if $Y$ is smooth, or more generally if $Y(\mathbb{R})$ is $\mathbb{F}_2$-smooth.
\end{itemize}
\end{thm}

In the complex case, the analogous intersection cohomology groups of $Y(\mathbb{C})$ arise as the hypercohomology groups of a certain complex of sheaves $\mathbf{IC}(Y(\mathbb{C}),k)$.
We also construct the groups $\mathscr{E}^i(Y(\mathbb{R}),\mathbb{F}_2)$ as the (hyper)cohomology groups of a complex of sheaves $\mathscr{E}(Y_\mathbb{R})$ in the constructible derived category of $Y(\mathbb{R})$ with $\mathbb{F}_2$ coefficients.

\begin{thm}\label{Thm:Real geo extension}
   Let $Y$ be a $d$-dimensional real variety, with nonempty smooth locus $Y^{sm}(\mathbb{R})$. We define the groups of Theorem \ref{Thm:coh of real geo ext} as the hypercohomology of a complex of sheaves $\mathscr{E}(Y_\mathbb{R})$ on $Y(\mathbb{R})$:\[\mathscr{E}^i(Y(\mathbb{R}),\mathbb{F}_2):=\mathbb{H}^i\mathscr{E}(Y_\mathbb{R})\]
   We call this complex $\mathscr{E}(Y_\mathbb{R})$ the real geometric extension on $Y$, and it has the following properties:

\begin{enumerate}
    \item For any resolution of singularities $f:X\to Y$, $\mathscr{E}(Y_\mathbb{R})$ is a direct summand of $Rf_*\mathbf{1}_{X(\mathbb{R})}$, the (derived) pushforward of the constant sheaf on $X(\mathbb{R})$.
    \item The complex $\mathscr{E}(Y_\mathbb{R})$ is isomorphic to its Verdier dual up to a shift by the dimension $d$:
    \[\mathbb{D}\mathscr{E}(Y_\mathbb{R})\cong \mathscr{E}(Y_\mathbb{R})[d]\]
    \item The cohomology sheaves of $\mathscr{E}(Y_\mathbb{R})$ are constructible and its (compactly supported) cohomology groups are finite dimensional, as is the cohomology of its stalks and costalks.
\end{enumerate}
\end{thm}

The duality in the cohomology groups of $\mathscr{E}(Y_\mathbb{R})$ is induced from the choice of isomorphism $\mathscr{E}(Y_\mathbb{R})\cong \mathbb{D}\mathscr{E}(Y_\mathbb{R})[-d]$ in the derived category, so induces the general version of Poincaré duality using compactly supported cohomology:
\[\mathbb{H}^i(\mathscr{E}(Y_\mathbb{R}))\cong \mathbb{H}^{d-i}_c(\mathscr{E}(Y_\mathbb{R}))^\vee\]
We warn the reader that this choice of duality for $\mathscr{E}(Y_\mathbb{R})$ is a genuine choice however, even in this $\mathbb{F}_2$ setting where fundamental classes are canonical.

The first property of Theorem \ref{Thm:Real geo extension} implies that this sheaf $\mathscr{E}(Y_\mathbb{R})$ (and thus its cohomology) occurs within the cohomology of any resolution.

\begin{cor}\label{Cor:fibre minimality}
Let $f:X\to Y$ be any resolution of singularities of $Y$. Then we have an injection of vector spaces:\[\mathscr{E}^i(Y(\mathbb{R}),\mathbb{F}_2)\hookrightarrow H^i(X(\mathbb{R}),\mathbb{F}_2)\]
Similarly, let $y\in Y(\mathbb{R})$ and let $F_y:=f^{-1}(y)$ be the fibre over $y$. Then the inclusion of the cohomology of the stalk of $\mathscr{E}(Y_\mathbb{R})$ at $y$ provides a lower bound for the Betti numbers of $F_y(\mathbb{R})$:
\[\dim H^i(i_y^* \mathscr{E}(Y_\mathbb{R}))\leq \dim H^i(F_y(\mathbb{R}),\mathbb{F}_2)\]
\end{cor}

In some cases, we may explicitly describe this complex of sheaves $\mathscr{E}(Y_\mathbb{R})$. For an irreducible variety $Y$ with nonempty smooth locus $Y^{sm}$, we say a resolution of singularities $f:X\to Y$ is small if the only irreducible component of $X\times_Y X$ of dimension at least $d$ is the closure of $\Delta(Y^{sm})$.

\begin{prop}
Let $f:X\to Y$ be a small resolution of $Y$. Then we may recognise this sheaf $\mathscr{E}(Y_\mathbb{R})$ as the derived pushforward of the constant sheaf along this resolution: \[\mathscr{E}(Y_\mathbb{R})\cong Rf_*\mathbf{1}_{X(\mathbb{R})}.\]
In particular, the cohomology of $\mathscr{E}(Y_\mathbb{R})$ recovers the cohomology of \emph{any} small resolution, if one exists:
\[\mathscr{E}^i(Y(\mathbb{R}),\mathbb{F}_2)\cong H^i(X(\mathbb{R}),\mathbb{F}_2).\]
\end{prop}

\begin{remark}
   That the mod two cohomology of the real points of any small resolution is independent of the choice of resolution was an original motivation for real intersection cohomology. This independence has been shown through the prior constructions of real intersection cohomology in \cite{Realintersectionhomology} and \cite{VanHamel}, and via different methods by Totaro in \cite{Totaro2000}.
\end{remark}
For suitable isolated singularities, we may also recognise $\mathscr{E}(Y_\mathbb{R})$ as the pushforward along a resolution.

\begin{prop}
    Let $Y$ have an isolated singularity at $y\in Y(\mathbb{R})$, such that the blowup $\pi:Bl_yY\to Y$ has smooth exceptional divisor $E$. Assume that $E(\mathbb{R})$ is nonempty, and that the real normal bundle of $E(\mathbb{R})$ in the real points of this blowup is topologically trivial. Then we have an isomorphism: \[\mathscr{E}(Y_\mathbb{R})\cong R\pi_*\mathbf{1}_{Bl_y Y}\]
    As such, for any resolution $f:X\to Y$, the fibre $F_y=f^{-1}(y)$ contains the cohomology of $E(\mathbb{R})$ as a direct summand, so:
    \[\dim H^i(E(\mathbb{R}),\mathbb{F}_2) \leq \dim H^i(F_y(\mathbb{R}),\mathbb{F}_2)\]
    Since $E(\mathbb{R})$ is a nonempty, compact manifold of codimension one in the blowup, this fibre $F_y(\mathbb{R})$ has dimension $d-1$.
\end{prop}

\begin{remark}
The real threefold $V(x^2+y^2+z^2-w^2)\subset \mathbb{A}^4$ provides an example of the above proposition; it does not have a small resolution over the reals. It does have two small resolutions over $\mathbb{C}$ however, via the isomorphism with the cone on $\mathbb{P}^1\times \mathbb{P}^1$ in the Segre embedding. This pair of resolutions is swapped by the Galois action under this isomorphism, so neither descends to $\mathbb{R}$.
\end{remark}

In these special cases, $\mathscr{E}(Y_\mathbb{R})$ is equal to the pushforward of the constant sheaf from a particular resolution. Schubert varieties provide another class of varieties where these real geometric extensions may be described, for different reasons. Let $G$ be a split reductive group over $\mathbb{R}$, with Borel $B$, and flag variety $\mathscr{F}:=G/B$. The closures $X(w)$ of the $B$-orbits on $\mathscr{F}$ are the Schubert varieties, indexed by elements of the Weyl group $W$ of $G$. The parity sheaves of Juteau-Mautner-Williamson \cite{ParitysheavesJMW} (with $\mathbb{F}_2$ coefficients) are a distinguished class of sheaves on $\mathscr{F}(\mathbb{C})$, and each $X(w)(\mathbb{C})$ has a unique indecomposable parity sheaf $\mathscr{E}(X(w)_\mathbb{C})$ extending the (degree zero) constant sheaf over the smooth locus. These parity sheaves admit an intrinsic parity vanishing description, and they may also be characterised as the mod two geometric extensions of \cite{HoneWilliamson2025}, the minimal dense summands occurring in pushforwards for resolutions of singularities. We may describe the category of real geometric extensions on Schubert varieties using parity sheaves on the corresponding complex varieties.

\begin{thm}
If $X(w)$ is a Schubert variety, then the cohomology groups are isomorphic to those of the corresponding parity sheaf, divided by two:
\[\mathscr{E}^i(X(w)(\mathbb{R}),\mathbb{F}_2):=\mathbb{H}^i\mathscr{E}(X(w)_\mathbb{R})\cong \mathbb{H}^{2i}\mathscr{E}(X(w)_\mathbb{C})\]
Similarly, for $u\in W$, with $i_u:uB/B\to \mathscr{F}$, the stalk and costalk of $\mathscr{E}({X(w)}_\mathbb{R})$ at $u$ are isomorphic to those of the corresponding parity sheaf, divided by two:
\[H^i(i_u^*\mathscr{E}(X(w)_\mathbb{R}))\cong H^{2i}(i_u^*\mathscr{E}(X(w)_\mathbb{C}))\]
    \[H^i(i_u^!\mathscr{E}(X(w)_\mathbb{R}))\cong H^{2i}(i_u^!\mathscr{E}(X(w)_\mathbb{C}))\]
The above isomorphisms follow from an equivalence of categories \[\mathrm{Geo}(\mathscr{F}(\mathbb{R}),\mathbb{F}_2)\cong \mathrm{Par}_{ev}(\mathscr{F}(\mathbb{C}),\mathbb{F}_2)\] which doubles degrees, where $\mathrm{Geo}(\mathscr{F}(\mathbb{R}),\mathbb{F}_2)$ is the additive category of sums and shifts of the sheaves $\mathscr{E}(X(w)_\mathbb{R})$ on the real flag manifold, and $\mathrm{Par}_{ev}(\mathscr{F}(\mathbb{C}),\mathbb{F}_2)$ the corresponding category of even parity sheaves. 
\end{thm}

We prove the equivalence in the above theorem by taking advantage of the motivic simplicity of Schubert varieties and their resolutions; they are paved by affine spaces. For any variety $Y$, we also define a variant $\mathscr{E}_{num}(Y_\mathbb{R})$ of our primary construction using the real cycle class map from Chow groups to Borel-Moore homology \cite{BorelHaefliger}. This $\mathscr{E}_{num}(Y_\mathbb{R})$ provides a Zariski-local version of $\mathscr{E}(Y_\mathbb{R})$.

\begin{prop}
Let $Y$ be a real variety of dimension $d$ with smooth real point. There exists a complex of sheaves $\mathscr{E}_{num}(Y_\mathbb{R})$ on $Y(\mathbb{R})$ with the following properties:
\begin{itemize}
     \item The sheaf $\mathscr{E}_{num}(Y_\mathbb{R})$ is Zariski-local, for $U\subset Y$ open, the restriction of $\mathscr{E}_{num}(Y_\mathbb{R})$ gives $\mathscr{E}_{num}(U_\mathbb{R})$:\[\mathscr{E}_{num}(Y_\mathbb{R})|_{U(\mathbb{R})}\cong \mathscr{E}_{num}(U_\mathbb{R})\]
     \item The complex $\mathscr{E}_{num}(Y_\mathbb{R})$ is $d$-shifted Verdier self-dual, its cohomology sheaves are constructible, and the real geometric extension is a direct summand of $\mathscr{E}_{num}(Y_\mathbb{R})$.

    \item For any resolution $f:X\to Y$, $\mathscr{E}_{num}(Y_\mathbb{R})$ is a direct summand of $f_*\mathbf{1}_{X(\mathbb{R})}$ such that the idempotent cutting it out is in the image of the real cycle class map: \[\mathrm{CH}_d(X\times_Y X)\to H^{\mathrm{BM}}_d(X(\mathbb{R})\times_{Y(\mathbb{R})}X(\mathbb{R}),\mathbb{F}_2) \to \End(f_*\mathbf{1}_{X(\mathbb{R})})\]
    
\end{itemize}

\end{prop}

This variant of the construction has better formal properties, but is less explicit due to the difficulty computing  image of the cycle class map. While we expect examples to exist, we don't know of any cases where these two constructions differ.
\subsection{Comparison with IC and the decomposition theorem}

The context and motivation for these constructions is the theory of intersection cohomology for complex varieties. In this context, for $Y$ an irreducible complex variety, the rational intersection cohomology sheaf arises in two distinct ways:

\begin{enumerate}
     \item[Intrinsic)] The sheaf $\mathbf{IC}(Y(\mathbb{C}),\mathbb{Q})$ is the unique shifted Verdier self-dual sheaf on $Y(\mathbb{C})$ extending the constant sheaf on the smooth locus satisfying the perversity bound on the cohomology of its stalks \cite[\S 2]{deCataldo2009decomposition}. More structurally, it is the unique simple object of dense support in the abelian subcategory of perverse sheaves within $D^b_c(Y(\mathbb{C}),\mathbb{Q})$.
    
    \item[Extrinsic)] The sheaf $\mathbf{IC}(Y(\mathbb{C}),\mathbb{Q})$ is the unique dense summand of $f_*\mathbf{1}_{X(\mathbb{C})}$ for any resolution of singularities $f:X\to Y$. 
   
\end{enumerate}

That these two constructions describe the same object is a corollary of the decomposition theorem of Beilinson-Bernstein-Deligne-Gabber \cite{BBD}, a deep result on the topology of algebraic maps (see \cite{deCataldo2009decomposition} for an excellent survey). With positive-characteristic coefficients, the decomposition theorem fails, and the intersection cohomology is not necessarily this dense summand of $f_*\mathbf{1}_{X(\mathbb{C})}$ for any resolution.

The real geometric extension is a direct generalisation of this extrinsic definition, and thus carries some key differences from intersection cohomology. For instance, $\mathscr{E}(Y_\mathbb{R})$ may have nonscalar automorphisms, so the choice of autoduality on $\mathscr{E}(Y_\mathbb{R})$ does not reduce to a fundamental class. While the lack of a perverse t-structure makes $\mathscr{E}(Y_\mathbb{R})$ harder to describe, its stalks are not constrained by the perversity bound and can carry more information about resolutions (see Proposition \ref{Prop:real blowup minimal}).

In view of the prior definitions of real intersection cohomology in \cite{Realintersectionhomology}, \cite{VanHamel}, we view the groups $\mathscr{E}^i(Y(\mathbb{R}),\mathbb{F}_2)$ as providing an alternative solution to the original problem of Goresky-MacPherson, taking a more extrinsic approach.

\subsection{Comparison to existing constructions}
This problem of real intersection cohomology has been previously addressed by  McCrory-Parusiński in \cite{Realintersectionhomology} and van Hamel in \cite{VanHamel}. These alternate constructions share key features with ours: they each give the correct cohomology groups for real varieties admitting a small resolution, provide a factorisation of the cycle class map, and recover ordinary cohomology for smooth varieties. In other key respects these existing constructions are quite different from ours.
\\
The approach of McCrory-Parusiński is explicitly geometric and parallels the original definition of Goresky-MacPherson. For a real algebraic variety $Y$ (or semialgebraic set) they define a complex of flasque sheaves by taking closed semialgebraic subsets with mod two relations \[[A]+[B]=[A\cup B] + [A\cap B].\] For a suitable stratification $S$ of $Y(\mathbb{R})$, they define a subcomplex by imposing intersection-dimension restrictions on chains, and show the resulting cohomology groups are independent of the stratification chosen. The cohomology groups of this complex are their real intersection cohomology groups. These groups come with a canonical intersection pairing, and they show that this is perfect for isolated singularities, but is not perfect in general \cite{MCCRORY2020107050}. It is also not known in general whether their groups are finitely generated.
\\
The approach of van Hamel is based on viewing $Y(\mathbb{R})$ as the fixed points of the conjugation involution within $Y(\mathbb{C})$. Working on a general pseudomanifold with involution, one has a derived category of $C_2$-equivariant sheaves, and the equivariant intersection cohomology sheaf $\mathbf{IC}_{C_2}(Y(\mathbb{C}),\mathbb{F}_2)$. The group cohomology ring of $C_2$ then acts on the equivariant cohomology groups of this sheaf, and the (total) real intersection cohomology is defined as \[\bigg{(}\bigoplus_{i\in \mathbb{Z}}\mathbb{H}^i\mathbf{IC}_{C_2}(Y(\mathbb{C}),\mathbb{F}_2)\bigg{)}/(1-\eta)\] for $\eta\in H^1(C_2,\mathbb{F}_2)$. As $1-\eta$ is not homogeneous, this total group does not carry a grading, but the $C_2$-equivariant localisation theorem shows that this depends only on the restriction to the fixed points.

In the language of Treumann's Smith theory for sheaves \cite{Treumann2019}, this may be described sheaf theoretically as taking the Smith restriction of intersection cohomology to the fixed points. The self-duality and finiteness properties of van Hamel's construction are then implied by the good properties of Smith restriction for sheaves.

We view our construction as defining the natural generalisation of geometric extensions and parity sheaves to the real setting, complementing these existing constructions of real intersection cohomology.
\subsection{Summary of proof}

The strategy for constructing $\mathscr{E}(Y_\mathbb{R})$ follows that of \cite{HoneWilliamson2025}, \cite{Mcnamara}, incorporating a simplification of Bezrukavnikov. One needs to show that there is a certain direct summand occurring in the (derived) pushforward of the constant sheaf along \emph{any} resolution of singularities. The existence of this sheaf $\mathscr{E}(Y_\mathbb{R})$ is a consequence of the following three facts.

\begin{enumerate}
    \item The derived category of sheaves is Krull-Schmidt, so for any map $f:X\to Y$, the pushforward $f_*\mathbf{1}_{X(\mathbb{R})}$ decomposes uniquely into a finite direct sum of indecomposable objects.
    \item For any proper birational map $g:X'\to X$ of smooth varieties, the induced map \[\mathbf{1}_{X(\mathbb{R})}\to g_*\mathbf{1}_{X'(\mathbb{R})}\] is a split monomorphism in $D^b_c(X(\mathbb{R}),\mathbb{F}_2)$.
    \item Any two resolutions of $Y$ are dominated by a third.
\end{enumerate}
To deduce the existence of $\mathscr{E}(Y_\mathbb{R})$ from these properties, note that if \[X'\xrightarrow{g} X\xrightarrow{f} Y\] are proper birational with $X$ and $X'$ smooth, then $g_*\mathbf{1}_{X'(\mathbb{R})}\cong \mathbf{1}_{X(\mathbb{R})}\oplus Q$ for $Q$ a sheaf of non-dense support. So we have \[(g\circ f)_*\mathbf{1}_{X'(\mathbb{R})}\cong f_*\mathbf{1}_{X(\mathbb{R})}\oplus f_*Q\] where $f_*Q$ has support of dimension less than $d$. The uniqueness of decompositions of summands in Krull-Schmidt categories implies the sum of the summands of full dimensional support is independent of the resolution, by the domination property for resolutions.

To construct the motivic variant $\mathscr{E}_{num}(Y_\mathbb{R})$, one adapts the previous proof, noting that the splittings of Corollary \ref{Cor:Splitting} are induced by cycles under the real cycle class map.

Once this independence is known, the other properties of $\mathscr{E}(Y_\mathbb{R})$ are simple consequences of the derived formalism.

\subsection{Structure of this paper}
In Section \ref{Sec:background} we recall some of the background notions we will need. We first recall the sheaf theoretic approach to cohomology via the six functor formalism of constructible sheaves, then recall some results on the topology of real algebraic varieties. In Section \ref{Sec:Construction} we give the construction of the real geometric extension $\mathscr{E}(Y_\mathbb{R})$, its variant $\mathscr{E}_{num}(Y_\mathbb{R})$ and verify its basic properties. In Section \ref{Sec:Examples and applications} we give examples and applications of this construction, and in Section \ref{Sec:Further questions} we conclude with some natural further questions.\\
\par
\emph{Acknowledgments:}
We would like to thank Geordie Williamson for valuable discussions and encouragement in this project, and Elden Elmanto for his infectious enthusiasm towards the real realisation of motives. We were supported by the Danish National Research Foundation through the Copenhagen Centre for Geometry and Topology (DNRF151). The contents of this paper are an adaptation of part of the author's PhD thesis.

\section{Sheaf theory on real algebraic varieties}\label{Sec:background}

In this section we recall some of the background theory we will need for our arguments. Real algebraic varieties are well-behaved from a point-set topological perspective, allowing the use of the six functor formalism of constructible complexes of $\mathbb{F}_2$ sheaves to understand their cohomology. As this formalism entails some categorical overhead, we begin with a quick introduction, with examples that will be useful in this paper. In the remainder of this section, we recall some standard properties of real algebraic varieties we will need. The simple Corollary \ref{Cor:Splitting} is the crucial result of this section needed for our later arguments.
\subsection{Cohomology and sheaf theory}
For a suitably nice topological space $X$, the cohomology of $X$ may be computed as the sheaf cohomology of the constant sheaf. Working in the derived category of all sheaves on $X$, one may use nonconstant sheaves to get at more refined geometric information. This framework is especially relevant for singular spaces, organising various cohomological constructions, and will be the home of our real geometric extension $\mathscr{E}(Y_\mathbb{R})$.

In this section we give a brief introduction to this formalism, suited for the purposes of this paper. For the cleanest setup, one needs some point set assumptions on the topological spaces involved, so for the rest of this section we assume the following properties of our spaces:
\begin{itemize}
    \item All spaces are locally compact Hausdorff of finite cohomological dimension.
    \item All maps $f:X\to Y$ are stratifiable, so pushforwards preserve constructibility.
\end{itemize}
To each space, we associate a category \[S_X:=D^b_c(X,\mathbb{F}_2),\] the derived category of complexes of sheaves of $\mathbb{F}_2$ vector spaces on $X$ with constructible cohomology sheaves. We will be using this category in a primarily formal way, and will refer to objects of this category as sheaves on $X$. We refer the reader to \cite{kashiwara2002sheaves} for a construction of these categories and the numerous functors between them. 

This category $S_X$ is additive, idempotent complete, and is triangulated, so has a shift functor $[1]$. For any map $f:X\to Y$, we have four induced functors between these sheaf categories:
\[\begin{tikzcd}
    S_X\arrow[rr,"f_*",shift left=1]\arrow[rr,"{f_!}"',shift right=1]&&S_Y \arrow[ll,bend left,"f^!"]\arrow[ll,bend right,"f^*"']
\end{tikzcd}\]
These come in two adjoint pairs $f^*\dashv f_*$ and $f_!\dashv f^!$, and the (co)units of adjunction supply many of the canonical maps between (co)homology groups of our spaces.

\begin{warning}
All of our functors are ``derived'' in the sense that they are defined between the derived sheaf categories. While $f_*$, $f_!$, and $f^*$ are derived from corresponding functors on abelian sheaves, $f^!$ is not, it is defined only at the derived level as a right adjoint to $f_!$.
\end{warning} 
The key properties of these functors may be seen in the example of the terminal map $t:X\to \ast$ from $X$ to the point.

\begin{ex}[The terminal map]\label{Ex:terminal map}
For the one-point space, $S_{\ast}$ is the bounded derived category of $\mathbb{F}_2$ vector spaces, and for $\mathbf{1}$ the one-dimensional space in degree zero, the cohomology of an object $\mathcal{F}:=\mathcal{F}^\bullet$ is given by \[H^i(\mathcal{F})=\hom_{S_\ast}(\mathbf{1},\mathcal{F}[i])\]
For the terminal map $t:X\to \ast$, the different pushforwards and pullbacks of $\mathbf{1}$ yield four different (co)homological measurements of $X$:
\begin{align*}
H^i(X,\mathbb{F}_2)\cong& \hom_{S_\ast}(\mathbf{1},{t}_*t^*\mathbf{1}[i]), \\
H^i_c(X,\mathbb{F}_2)\cong& \hom_{S_\ast}(\mathbf{1},{t}_!t^*\mathbf{1}[i]), \\
H_i(X,\mathbb{F}_2)\cong& \hom_{S_\ast}(\mathbf{1},{t}_!t^!\mathbf{1}[-i]), \\
H^{\mathrm{BM}}_i(X,\mathbb{F}_2)\cong& \hom_{S_\ast}(\mathbf{1},{t}_*t^!\mathbf{1}[-i]),
\end{align*}
where $H^i_c(X,\mathbb{F}_2)$ is compactly supported cohomology, and $H^{\mathrm{BM}}_i(X,\mathbb{F}_2)$ is Borel-Moore (or noncompact) homology.
For a general sheaf $\mathcal{F}$ in $S_X$, these terminal pushforwards are its (hyper)cohomology, potentially with compact supports:
\begin{align*}
    \hom_{S_\ast}(\mathbf{1},t_*\mathcal{F}[i]):=&\mathbb{H}^i(\mathcal{F})\\
    \hom_{S_\ast}(\mathbf{1},t_!\mathcal{F}[i]):=&\mathbb{H}^i_c(\mathcal{F})
\end{align*}

\end{ex}

\begin{defi}
We define the constant and dualising sheaves on $X$ to be the two pullbacks of $\mathbf{1}$ along this terminal map:
\begin{align*}
    \mathbf{1}_X:=&t^*\mathbf{1}\\
    \omega_X:=&t^!\mathbf{1}
\end{align*}
The above example shows that the cohomology of these sheaves computes the four flavours of cohomology of $X$.
\end{defi}
In addition to the units and counits of adjunction, certain topological properties of maps yield natural transformations between these functors. For a topologically separated map $f:X\to Y$ we have a comparison morphism of functors \[f_!\to f_*\] and this is an isomorphism if $f$ is topologically proper. Similarly, for an open inclusion (or more generally local homeomorphism) $j:U\to X$, we have a canonical isomorphism: \[j^!\xrightarrow{\sim} j^*.\] For an open subset $U$ of $X$, we call  $j^!\mathcal{F}\cong j^*\mathcal{F}$ the restriction of $\mathcal{F}$ to $U$, and write it as $\mathcal{F}|_U$. For closed inclusion $i:Z\to X$, the canonical (co)units yield isomorphisms: \[i^*i_*\xrightarrow{\sim}\text{Id}_{S_Z}\xrightarrow{\sim} i^!i_!\] and this implies $i_!\cong i_*$ is fully faithful. This gives a well-behaved notion of the support $supp(\mathcal{F})$ for a sheaf $\mathcal{F}$ in $S_X$, the smallest closed set $Z$ such that $\mathcal{F}\cong i_*\mathcal{F}'$ for some $\mathcal{F}'$ on $Z$.
\begin{ex}[Recollement sequences]\label{Ex:Recollement sequences}
The (co)units of adjunction interact with the triangulated structure of $S_X$ to produce the usual long exact sequences of the various flavours of (co)homology. For a sheaf $\mathcal{F}$ in $S_X$, and an open-closed decomposition $X=U\cup Z$, with open inclusion $j:U\to X$, closed inclusion $i:Z\to X$, the (co)units of adjunction supply two functorial triangles in $S_X$:
\begin{align*}
    i_!i^!\mathcal{F}&\to \mathcal{F}\to j_*j^*\mathcal{F}\to\\
    j_!j^!\mathcal{F}&\to \mathcal{F}\to i_*i^*\mathcal{F}\to
\end{align*}
Applying $t_*$ to the second sequence at $\mathcal{F}=\mathbf{1}_X$, one recovers the long exact sequence in cohomology, and gives the identification of $H^*(X,Z,\mathbb{F}_2)$ with $H^*(t_*j_!\mathbf{1}_U)$. Similarly, applying $t_*$ to the first sequence with $\mathcal{F}=\omega_X$, using $i_!\cong i_*$ and $j^!\cong j^*$ gives the localisation sequence in Borel-Moore homology:
\[H^{\mathrm{BM}}_*(Z,\mathbb{F}_2)\to H^{\mathrm{BM}}_*(X,\mathbb{F}_2)\to H^{\mathrm{BM}}_*(U,\mathbb{F}_2)\to H^{\mathrm{BM}}_{\ast-1}(Z,\mathbb{F}_2)\to\]
\end{ex}
This category $S_X$ carries a tensor product $\otimes$ and an internal $\Hom$, though for our purposes we will only need a special case of these, the Verdier duality functor: \[\mathbb{D}_X(\_):=\Hom(\_,\omega_X)\]
Our constructibility assumption implies that this functor is a duality, and it exchanges $!$ and $*$:
\begin{align*}
\mathbb{D}_X^2\cong\ & \text{Id}_{S_X}\\
    f_*\circ \mathbb{D}_X\cong\ & \mathbb{D}_Y\circ f_!\\
    f^*\circ \mathbb{D}_Y\cong\ & \mathbb{D}_X\circ f^!
\end{align*}
We may understand the utility of this abstract formalism in the case where our spaces are manifolds.
\begin{ex}[Manifolds]\label{ex:Manifolds}
For a $d$-dimensional manifold $M$, a fundamental class in (mod two) Borel-Moore homology induces an isomorphism $\mathbf{1}_M\to \omega_M[-d]$. Pushing forward to the point, this induces Poincaré duality:
\[H^i_c(M,\mathbb{F}_2)\cong \hom(\mathbf{1},t_!\mathbf{1}_M[i])\cong \hom(\mathbf{1},t_!\omega_M[i-d])\cong H_{d-i}(M,\mathbb{F}_2)\]
Similarly, for a proper map of manifolds $f:M\to N$ of dimensions $d_M$, $d_N$ respectively, the composition \[f_*\mathbf{1}_M \cong f_*\omega_M[-d_M]\cong f_!\omega_M[-d_M]\xrightarrow{\epsilon} \omega_N[-d_M]\cong \mathbf{1}_N[d_N-d_M]\] induces the usual Umkehr map in cohomology.

In the special case of a submanifold inclusion $i:C\to M$ of codimension $c$, the composition \[\mathbf{1}_C\cong i^*i_*\mathbf{1}_C\to i^*\mathbf{1}_M[c]\cong \mathbf{1}_C[c]\] classifies the top Stiefel-Whitney class of the normal bundle of $C$ in $M$.
\end{ex}

The final property of this six functor formalism we will need is base change. For a pullback diagram \begin{equation}\label{diagram:pullback}
    \begin{tikzcd}
    Z\times_Y X\arrow[r,"g'"]\arrow[d,"f'"]&X\arrow[d,"f"]\\
    Z\arrow[r,"g"]&Y
\end{tikzcd}
\end{equation}
base change refers to the following canonical isomorphisms of functors:
\[ g^*\circ f_!\cong (f')_!\circ (g')^*,\]
\[g^!\circ f_*\cong (f')_*\circ (g')^!.\]
In the special case of the constant sheaf, when $f$ is proper, and $Z$ is the inclusion $i_y$ of a point of $Y$, this first isomorphism allows one to interpret the stalk of the pushforward $f_*\mathbf{1}_X$ as the cohomology of the fibre $F_y$ over $y$:
\[i_y^*(f_*\mathbf{1}_X)\cong i_y^*(f_!\mathbf{1}_X)\cong (f|_{F_y})_!\mathbf{1}_{F_y}\cong (f|_{F_y})_*\mathbf{1}_{F_y}\cong H^*(F_y,\mathbb{F}_2)\]
Our final example of base change is that of the convolution isomorphism.
\begin{ex}[Convolution isomorphism]\label{Ex:conv iso}
With respect to the pullback square \ref{diagram:pullback}, consider the special case where $g$ is proper and $X$ is an oriented manifold of dimension $d$ with isomorphism $\mathbf{1}_X\cong \omega_X[-d]$. Then we may describe the maps between $g_*\mathbf{1}_Z$ and $f_*\mathbf{1}_X$ using the Borel-Moore homology of $Z\times_Y X$ via the following isomorphisms:
\begin{align*}
\hom(g_*\mathbf{1}_Z,f_*\mathbf{1}_X)&\cong\hom(g_!\mathbf{1}_Z,f_*\mathbf{1}_X)&&\!\!\!\!\!\!\!\!\!\!\!\!\!\!\!\!\!\!\!\cong \hom(\mathbf{1}_Z,g^!f_*\mathbf{1}_X)\\
&&&\!\!\!\!\!\!\!\!\!\!\!\!\!\!\!\!\!\!\!\cong \hom(\mathbf{1}_Z,(f')_*(g')^!\mathbf{1}_X)\\
&&&\!\!\!\!\!\!\!\!\!\!\!\!\!\!\!\!\!\!\!\cong\hom((f')^*\mathbf{1}_Z,(g')^!\mathbf{1}_X)\\
&&&\!\!\!\!\!\!\!\!\!\!\!\!\!\!\!\!\!\!\!\cong \hom((f')^*\mathbf{1}_Z,(g')^!\omega_X[-d])\\
&&&\!\!\!\!\!\!\!\!\!\!\!\!\!\!\!\!\!\!\!\cong \hom(\mathbf{1}_{Z\times_Y X},\omega_{Z\times_Y X}[-d])\\
&&&\!\!\!\!\!\!\!\!\!\!\!\!\!\!\!\!\!\!\!\cong H_d^{\mathrm{BM}}(Z\times_Y X,\mathbb{F}_2)
\end{align*}
This isomorphism transforms the problem of constructing morphisms between these sheaves into finding homology classes on the fibre product.

In the special case when $Y$ is a point, this is just the Künneth isomorphism and Poincaré duality: \begin{align*}
    \hom_{S_\ast}(H^*(Z),H^*(X))&\cong \bigoplus_i H^i(Z)^\vee \otimes H^i(X)\\
    &\cong\bigoplus_i H_i(Z)\otimes H^{\mathrm{BM}}_{d-i}(X)\cong H^{\mathrm{BM}}_d(Z\times X).
\end{align*}
\end{ex}
\subsection{Topological properties of real algebraic varieties}

In this section we review some topological properties of real algebraic varieties we will need. We refer the reader to \cite{bochnak1998real} for an introduction to real varieties.

A real variety $X$ is a reduced and irreducible separated scheme of finite type over $\mathbb{R}$. This has a locally compact Hausdorff topological space of real points $X(\mathbb{R})$, and algebraic maps between real algebraic varieties induce stratifiable maps of their associated topological spaces \cite[Ch. 9]{bochnak1998real}. We will therefore work within the formalism of constructible $\mathbb{F}_2$ sheaves of the previous section. From here, all cohomology will be taken with $\mathbb{F}_2$ coefficients when unspecified, and we will use $X$ vs $X(\mathbb{R})$ to differentiate between the variety $X$ and its topological space of real points.

Real semialgebraic sets (such as $Y(\mathbb{R})$) have a well-behaved notion of dimension, and this agrees with cohomological dimension as topological spaces (see \cite{delfs1991homology}). For $Y$ a real variety, the dimension of $Y(\mathbb{R})$ may not agree with the (Krull) dimension of $Y$, and for singular varieties, the topological dimension may depend on the connected component of $Y(\mathbb{R})$ chosen. We will be avoiding these complications by restricting to the class of real varieties which have a smooth real point.

The following proposition gives the sheaf theoretic interpretation of this condition.

\begin{prop}\label{Prop:smooth real point}
Let $Y$ be a $d$-dimensional real algebraic variety, with $Y(\mathbb{R})_i$ a connected component of $Y(\mathbb{R})$. Then the following are equivalent:
    \begin{enumerate}
        \item $Y(\mathbb{R})_i$ contains a smooth real point.
        \item For some $y\in Y(\mathbb{R})_i$, we have $H^{-d}(i_y^*\omega_{Y(\mathbb{R})_i})\neq 0$. 
    \end{enumerate}
\end{prop}

\begin{proof}
    $1)\implies 2)$: Letting $y$ be this smooth point, the stalk of the dualising sheaf is one-dimensional, concentrated in degree $-d$.\\
    $2)\implies 1)$: If $Y^{sm}(\mathbb{R})$ were empty, then $Y(\mathbb{R})=(Y\setminus Y^{sm})(\mathbb{R})$, so has cohomological dimension $<d$, which contradicts the nonvanishing of \[H^d_{\{y\}}(Y(\mathbb{R}),\mathbb{F}_2)\cong H^d(i_y^!\mathbf{1}_{Y(\mathbb{R})_i})\cong H^{-d}(i_y^*\omega_{Y(\mathbb{R})_i})^\vee.\]
\end{proof}

\begin{lem}
Let $g:W\to X$ be a proper birational map between smooth $d$-dimensional real varieties, each with a smooth real point. Then the induced map on real points $g:W(\mathbb{R})\to X(\mathbb{R})$ is topologically proper, surjective, and an isomorphism over a topologically dense open subset of $X(\mathbb{R})$.
\end{lem}

\begin{proof}
The induced map on real points is topologically proper as it is proper over $\mathbb{C}$, and $X(\mathbb{R})$ and $W(\mathbb{R})$ are locally compact Hausdorff. As $g$ is birational, it is an isomorphism $g^{-1}(U)\to U$ for some $U$ open in $X$. The complement $Z=X\setminus U$ is of dimension $<d$, so $U(\mathbb{R})$ is dense in each component of the $d$-dimensional manifold $X(\mathbb{R})$. As the map is closed with dense image, it is surjective.
\end{proof}
\begin{remark}
    The surjectivity on real points may be false if $X$ is not assumed to be smooth. This can be seen from the isolated point $(0,0)$ on the singular plane curve $y^2=x^2(x-1)$ and its normalisation.
\end{remark}

A fundamental class of a $d$-dimensional manifold $M$ may be viewed as an isomorphism $\mathbf{1}_M\to \omega_M[-d]$. For a $d$-dimensional space $V$ with dense open submanifold $V^\circ$, a fundamental class of $V$ is defined to be a class in $H_d^{\mathrm{BM}}(V)$ restricting to a fundamental class of $V^\circ$. Sheaf-theoretically, this is a morphism \[\mathbf{1}_V\to \omega_V[-d]\] which restricts to an isomorphism over $V^\circ$. Real varieties admit unique, compatible fundamental classes, and we recall the short proof using resolution of singularities.

\begin{prop}
Let $Y$ be a $d$-dimensional real algebraic variety with smooth point. Then $Y(\mathbb{R})$ admits a unique fundamental class. For a proper birational map $f:Y'\to Y$ between such varieties, the pushforward of a fundamental class of $Y'(\mathbb{R})$ equals the fundamental class on $Y(\mathbb{R})$.
\end{prop}

\begin{proof}
We may construct such a class for $Y(\mathbb{R})$ by pushing forward along a resolution of singularities. For such a resolution $f:X\to Y$, a fundamental class \[[X(\mathbb{R})]:\mathbf{1}_{X(\mathbb{R})}\to \omega_{X(\mathbb{R})}[-d]\] of $X(\mathbb{R})$ pushes forward by the canonical adjunction maps and properness of $f$:\[\mathbf{1}_{Y(\mathbb{R})}\to f_*\mathbf{1}_{X(\mathbb{R})}\xrightarrow{f_*[X(\mathbb{R})]}f_* \omega_{X(\mathbb{R})}[-d]\xrightarrow{\sim} f_! \omega_{X(\mathbb{R})}[-d]\to  \omega_{Y(\mathbb{R})}[-d].\]
By base change, this map restricts to an isomorphism over the smooth locus $Y^{sm}(\mathbb{R})$, and supplies the desired orientation of $Y(\mathbb{R})$. Any smooth variety has a unique fundamental class as we are working over $\mathbb{F}_2$, and if $Z$ is the complement of the smooth locus in $Y$, then $Z(\mathbb{R})$ has dimension at most $d-1$. The long exact sequence \[0=H_d^{\mathrm{BM}}(Z(\mathbb{R}))\to H_d^{\mathrm{BM}}(Y(\mathbb{R}))\to H_d^{\mathrm{BM}}(Y^{sm}(\mathbb{R}))\] then shows that this lift of a fundamental class is unique. The uniqueness of this class then implies the desired compatibility under pushforward.
\end{proof}

This existence of fundamental classes implies the following crucial splitting for birational maps between smooth real varieties.

\begin{cor}\label{Cor:Splitting}
For $g:W\to X$ a proper birational morphism of smooth real varieties, the canonical map $\mathbf{1}_{X(\mathbb{R})}\to g_*\mathbf{1}_{W(\mathbb{R})}$ is split injective; there exists $Q\in D^b_c(X(\mathbb{R}),\mathbb{F}_2)$ such that \[g_*\mathbf{1}_{W(\mathbb{R})}\cong \mathbf{1}_{X(\mathbb{R})}\oplus Q.\]
\end{cor}

\begin{proof}
The compatibility of fundamental classes under pushforward implies the composition \[\mathbf{1}_{X(\mathbb{R})}\to g_*\mathbf{1}_{W(\mathbb{R})}\to g_!\omega_{W(\mathbb{R})}[-d]\to \omega_{X(\mathbb{R})}[-d]\] is the fundamental class of $X$, and is thus an isomorphism. The inverse of this isomorphism then yields the desired retract, and the splitting exists by idempotent completeness of the (additive) category of sheaves $D^b_c(X(\mathbb{R}),\mathbb{F}_2)$.
\end{proof}
\section{Construction of the real geometric extension}\label{Sec:Construction}

In this section we construct the real geometric extension $\mathscr{E}(Y_\mathbb{R})$ and its numerical motivic variant, and verify the properties stated in Theorem \ref{Thm:coh of real geo ext} and Theorem \ref{Thm:Real geo extension}.

\begin{thm}\label{C2.Realgeoextensions}
Let $Y$ be a $d$-dimensional real variety, with nonempty smooth locus $Y^{sm}(\mathbb{R})$. Then there exists a sheaf $\mathscr{E}(Y_\mathbb{R})$ in $D^b_c(Y(\mathbb{R}),\mathbb{F}_2)$ characterised by the following two properties:
\begin{itemize}
    \item This sheaf $\mathscr{E}(Y_\mathbb{R})$ extends the constant sheaf on $Y^{sm}(\mathbb{R})$, and has no proper direct summand with this property.
    \item For any resolution of singularities $f:X\to Y$, the sheaf $\mathscr{E}(Y_\mathbb{R})$ is a direct summand of $f_*\mathbf{1}_{X(\mathbb{R})}$.
\end{itemize}

\end{thm}

\begin{proof}
The constructible derived category of sheaves on $Y(\mathbb{R})$ has finite dimensional morphism spaces and is idempotent complete and additive, so is Krull-Schmidt (Lemma 5.2 of \cite{KRAUSE2015535}).
For any choice of resolution $f:X\to Y$, the Krull-Schmidt property implies we may find a direct sum decomposition \[f_*\mathbf{1}_{X(\mathbb{R})}\cong \mathscr{E}_f\oplus Q\] where $\mathscr{E}_f$ extends the constant sheaf over $Y^{sm}(\mathbb{R})$, and $\mathscr{E}_f$ has no proper direct summand with this property. The isomorphism class of $\mathscr{E}_f$ is defined independently of the chosen splitting; for a given Krull-Schmidt decomposition \[f_*\mathbf{1}_{X(\mathbb{R})}\cong \bigoplus^n_{i=1} F_i,\] $\mathscr{E}_f$ is isomorphic to the summand of those $F_i$ with $F_i|_{Y^{sm}(\mathbb{R})}\neq 0$. It remains to show that the isomorphism class of $\mathscr{E}_f$ is independent of the resolution $f$ chosen. Let $f_i:X_i\to Y$ for $i=1,2$ be two resolutions of singularities of $Y$, dominated by a third $f_3:X_3\to Y$. By Corollary \ref{Cor:Splitting}, we have that $(f_1)_*\mathbf{1}_{X_1(\mathbb{R})}$ and $(f_2)_*\mathbf{1}_{X_2(\mathbb{R})}$ are direct summands of  $(f_3)_*\mathbf{1}_{X_3(\mathbb{R})}$, so $\mathscr{E}_{f_1}$, viewed as a summand of $(f_3)_*\mathbf{1}_{X_3(\mathbb{R})}$, is isomorphic to $\mathscr{E}_{f_3}$. By symmetry, we conclude $\mathscr{E}_{f_1}\cong \mathscr{E}_{f_2}$, showing independence of the resolution.
\end{proof}

We call this sheaf $\mathscr{E}(Y_\mathbb{R})$ the real geometric extension on $Y$.
We now turn to verification of the properties of Theorem \ref{Thm:coh of real geo ext}, and we will assume for the rest of this section that $Y$ is a $d$-dimensional real variety with a smooth real point.

\begin{prop}
The cohomology of the sheaf $\mathscr{E}(Y_\mathbb{R})$ vanishes outside the interval $[0,d]$:
\[\mathscr{E}^i(Y(\mathbb{R}),\mathbb{F}_2):=\mathbb{H}^i\mathscr{E}(Y_\mathbb{R})=0 \text{ if } i>d \text{ or } i<0.\]
In addition, the stalks and costalks of $\mathscr{E}(Y_\mathbb{R})$ are finite-dimensional.
\end{prop}

\begin{proof}
The finiteness statements hold by constructibility, and the vanishing holds as $\mathbb{H}^i\mathscr{E}(Y_\mathbb{R})$ is a direct summand of $\mathbb{H}^if_*\mathbf{1}_{X(\mathbb{R})}\cong H^i(X(\mathbb{R}),\mathbb{F}_2)$.
\end{proof}

As a corollary of base change, we obtain that the size of the stalks of $\mathscr{E}(Y_\mathbb{R})$ are a lower bound for the fibre Betti numbers in any resolution.

\begin{cor}
For any resolution of singularities $f:X\to Y$, the cohomology of the stalk of $\mathscr{E}(Y_\mathbb{R})$ at $y$ provides a lower bound for the Betti numbers of the fibre $F_y=f^{-1}(y)$:
    \[\dim H^i(i_y^* \mathscr{E}(Y_\mathbb{R}))\leq \dim H^i(F_y(\mathbb{R}),\mathbb{F}_2)\]
\end{cor}

\begin{proof}
As the map $f$ is proper, base change implies this stalk of $\mathscr{E}(Y_\mathbb{R})$ at $y$ is a summand of $(f|_{F_y(\mathbb{R})})_*\mathbf{1}_{F_y(\mathbb{R})}$, which computes the cohomology of $F_y(\mathbb{R})$.
\end{proof}

\begin{prop}\label{propertiesrealgeoext}
The real geometric extension $\mathscr{E}(Y_\mathbb{R})$ on $Y(\mathbb{R})$ is a $d$-shifted self-dual complex of sheaves; there exists an isomorphism: \[\mathbb{D}\mathscr{E}(Y_\mathbb{R})\cong \mathscr{E}(Y_\mathbb{R})[d].\]
Its cohomology therefore satisfies Poincaré duality:    
\[\mathbb{H}^i(\mathscr{E}(Y_\mathbb{R}))\cong \mathbb{H}^{d-i}_c(\mathscr{E}(Y_\mathbb{R}))^\vee.\]
In particular, if $Y(\mathbb{R})$ is compact, we have the duality \[\mathscr{E}^i(Y(\mathbb{R}),\mathbb{F}_2):=\mathbb{H}^i(\mathscr{E}(Y_\mathbb{R}))\cong \mathbb{H}^{d-i}(\mathscr{E}(Y_\mathbb{R}))^\vee=:\mathscr{E}^{d-i}(Y(\mathbb{R}),\mathbb{F}_2)^\vee\]
\end{prop}

\begin{proof}
Let $f:X\to Y$ be any resolution of singularities of $Y$, and let  $\mathbf{1}_{X(\mathbb{R})}[d]\cong \omega_{X(\mathbb{R})}$ be the (shifted) fundamental class of $X(\mathbb{R})$. As $f$ is proper, this induces the isomorphism:
\[\mathbb{D}_{Y(\mathbb{R})}f_*\mathbf{1}_{X(\mathbb{R})}\cong \mathbb{D}_{Y(\mathbb{R})} f_!\mathbf{1}_{X(\mathbb{R})}\cong f_*\mathbb{D}_{X(\mathbb{R})}\mathbf{1}_{X(\mathbb{R})}\cong f_*\omega_{X(\mathbb{R})}\cong f_*\mathbf{1}_{X(\mathbb{R})}[d].\]
As taking duals and shifts commutes with restriction to an open subset, the isomorphism class of a minimal summand $\mathscr{E}(Y_\mathbb{R})$ of $f_*\mathbf{1}_{X(\mathbb{R})}$ extending $\mathbf{1}_{Y^{sm}(\mathbb{R})}$ over the smooth locus is $d$-shifted self-dual. This self-duality of $\mathscr{E}(Y_\mathbb{R})$ then implies the Poincaré duality statement via pushforward along the terminal map to a point.
\end{proof}

The duality of the above proposition is nonconstructive, and we do not know of a canonical description of it. Such a duality is not uniquely determined by a fundamental class over the smooth locus, and this fact reflects the non-perversity of $\mathscr{E}(Y_\mathbb{R})$. If $Z$ is the singular locus of $Y$, we may understand the ambiguity of this duality isomorphism using the following exact sequence, obtained from the recollement sequence of Example \ref{Ex:Recollement sequences} by taking maps from $\mathscr{E}(Y_\mathbb{R})$:
\[\hom(i^*\mathscr{E}(Y_\mathbb{R}),i^!\mathscr{E}(Y_\mathbb{R}))\to \hom(\mathscr{E}(Y_\mathbb{R}),\mathscr{E}(Y_\mathbb{R}))\to \hom(\mathbf{1}_{Y^{sm}(\mathbb{R})},\mathbf{1}_{Y^{sm}(\mathbb{R})})\]
In the case of an isolated singularity, the vanishing of this error term \[\hom(i^*\mathscr{E}(Y_\mathbb{R}),i^!\mathscr{E}(Y_\mathbb{R}))\] is equivalent to the cohomological supports of $i^*\mathscr{E}(Y_\mathbb{R})$ and $i^!\mathscr{E}(Y_\mathbb{R})$ being disjoint, which in the complex case is implied by the perversity condition. The induced comparison map from cohomology is independent of all choices however, and the factorisation of Theorem \ref{Thm:coh of real geo ext} holds for any choice of duality of $\mathscr{E}(Y_\mathbb{R})$.

\begin{prop}\label{Prop:factorise cap product}
For a choice of resolution $f:X\to Y$, and choice of splitting $f_*\mathbf{1}_{X(\mathbb{R})}\xrightarrow{\pi}\mathscr{E}(Y_\mathbb{R})$, the composition $\mathbf{1}_{Y(\mathbb{R})}\to f_*\mathbf{1}_{X(\mathbb{R})}\xrightarrow{\pi} \mathscr{E}(Y_\mathbb{R})$ is independent of $\pi$, yielding canonical comparison maps: \[\gamma_i:H^i(Y(\mathbb{R}),\mathbb{F}_2)\to \mathscr{E}^i(Y(\mathbb{R}),\mathbb{F}_2).\]
For complete $Y$, so $Y(\mathbb{R})$ is compact, for any choice of duality $\mathscr{E}(Y_\mathbb{R})\xrightarrow{\sim}\mathbb{D}\mathscr{E}(Y_\mathbb{R})[-d]$, the composition \[\mathbf{1}_{Y(\mathbb{R})}\to \mathscr{E}(Y_\mathbb{R})\xrightarrow{\sim}\mathbb{D}\mathscr{E}(Y_\mathbb{R})[-d]\to \omega_{Y(\mathbb{R})}[-d]\] induces a commutative diagram:
\[\begin{tikzcd}
H^i(Y(\mathbb{R}),\mathbb{F}_2)\arrow[r,"\cap \text{[}Y(\mathbb{R})\text{]} "]\arrow[d,"\gamma_i"]&H_{d-i}(Y(\mathbb{R}),\mathbb{F}_2)\\
\mathscr{E}^i(Y(\mathbb{R}),\mathbb{F}_2)\arrow[r,"\sim"]&\mathscr{E}^{d-i}(Y(\mathbb{R}),\mathbb{F}_2)^\vee\arrow[u,"\gamma_{d-i}^\vee"']
\end{tikzcd}\]

\end{prop}

\begin{proof}
The first map may be given intrinsically as the unique morphism $\mathbf{1}_{Y(\mathbb{R})}\to \mathscr{E}(Y_\mathbb{R})$ inducing an isomorphism over the smooth locus of $Y$. To see that a map with this property is unique, note first that any morphism $\mathbf{1}_{Y(\mathbb{R})}\to f_*\mathbf{1}_{X(\mathbb{R})}$ which is an isomorphism over the smooth locus of $Y$ corresponds to an endomorphism of $\mathbf{1}_{X(\mathbb{R})}$ by adjunction, which is an isomorphism over a dense open in $X(\mathbb{R})$, and is thus unique. The split inclusion \[\hom(\mathbf{1}_{Y(\mathbb{R})},\mathscr{E}(Y_\mathbb{R}))\to \hom(\mathbf{1}_{Y(\mathbb{R})},f_*\mathbf{1}_{X(\mathbb{R})})\] then implies the desired uniqueness of $\mathbf{1}_{Y(\mathbb{R})}\to \mathscr{E}(Y_\mathbb{R})$. Similarly, for any choice of duality, the composition \[\mathbf{1}_{Y(\mathbb{R})}\to \mathscr{E}(Y_\mathbb{R})\xrightarrow{\sim}\mathbb{D}\mathscr{E}(Y_\mathbb{R})[-d]\to \omega_{Y(\mathbb{R})}[-d]\] is the fundamental class of $Y(\mathbb{R})$, so the diagram commutes.
\end{proof}

\begin{prop}
The constant sheaf on $Y(\mathbb{R})$ is the real geometric extension if and only if $Y(\mathbb{R})$ is $\mathbb{F}_2$-smooth, the following are equivalent:
\begin{enumerate}
    
    \item The space $Y(\mathbb{R})$ is $\mathbb{F}_2$-smooth; the fundamental class is an isomorphism: \[\mathbf{1}_{Y(\mathbb{R})}\xrightarrow{\sim} \omega_{Y(\mathbb{R})}[-d].\]
    \item The canonical map from the constant sheaf to the real geometric extension on $Y$ is an isomorphism: \[\mathbf{1}_{Y(\mathbb{R})}\xrightarrow{\sim} \mathscr{E}(Y_\mathbb{R}).\]

\end{enumerate}
In this case, the comparison maps from cohomology are all isomorphisms:\[\gamma_i:H^i(Y(\mathbb{R}),\mathbb{F}_2)\xrightarrow{\sim} \mathscr{E}^i(Y(\mathbb{R}),\mathbb{F}_2)\]
\end{prop}

\begin{proof}
$1)\implies 2)$: The fundamental class of a resolution pushes forward to the fundamental class of $Y(\mathbb{R})$, providing a splitting of $\mathbf{1}_{Y(\mathbb{R})}\to f_*\mathbf{1}_{X(\mathbb{R})}$. Proposition \ref{Prop:smooth real point} then implies that all connected components $Y(\mathbb{R})_i$ are $d$-dimensional, yielding $\mathscr{E}(Y_\mathbb{R})\cong \mathbf{1}_{Y(\mathbb{R})}$.\\
$2)\implies 1)$: The autoduality of $\mathscr{E}(Y_\mathbb{R})$ implies that the fundamental class map is an isomorphism.
\end{proof}

\subsection{A motivic Zariski-local variant}

As the real geometric extension $\mathscr{E}(Y_\mathbb{R})$ is defined by the extraction of a direct summand, some functoriality properties of this construction are unclear. One natural question is whether the restriction of $\mathscr{E}(Y_\mathbb{R})$ to an open subset $U$ of $Y$ gives the corresponding real geometric extension on $U$. We do not know if this holds for $\mathscr{E}(Y_\mathbb{R})$, and this Zariski-locality was required in a recent application of geometric extensions to representation theory \cite{baine2026geometrytiltingcompositionseries}, so we include this alternate construction, which is Zariski-local.

\begin{defi}
Let $f:X\to Y$ be a resolution of singularities for $Y$. The numerical real geometric extension $\mathscr{E}_{num}(Y(\mathbb{R}))$ is a minimal summand of $f_*\mathbf{1}_{X(\mathbb{R})}$ extending the constant sheaf over $Y^{sm}(\mathbb{R})$ such that its associated idempotent is in the image of the real cycle class map:
\[\mathrm{CH}_d(X\times_Y X)\to H^{\mathrm{BM}}_d(X(\mathbb{R})\times_{Y(\mathbb{R})}X(\mathbb{R}),\mathbb{F}_2) \to \End(f_*\mathbf{1}_{X(\mathbb{R})})\]
\end{defi}

The proof that the isomorphism class of this object is well defined independently of the resolution follows the proof of Theorem \ref{Thm:Real geo extension}. One needs that the crucial splittings of Corollary \ref{Cor:Splitting} are induced by cycles, and has the Krull-Schmidt property since the image of this real cycle class map is a finite-dimensional $\mathbb{F}_2$-vector space. For more details on this alternate definition, we refer the reader to \cite{baine2026geometrytiltingcompositionseries}. The main upshot of this definition is the following Zariski-locality.

\begin{prop}
The restriction of the numerical real geometric extension on $Y(\mathbb{R})$ to $U$ open in $Y$ yields the numerical real geometric extension on $U(\mathbb{R})$:
\[\mathscr{E}_{num}(Y_\mathbb{R})|_{U(\mathbb{R})}\cong \mathscr{E}_{num}(U_\mathbb{R})\]
\end{prop}

\begin{proof}
Let $f:X\to Y$ be a resolution of $Y$, with cycle class map 
\[c_f:\mathrm{CH}_d(X\times_{Y} X)\to \End(f_*\mathbf{1}_{X(\mathbb{R})})\] and idempotent $e:f_*\mathbf{1}_{X(\mathbb{R})}\to f_*\mathbf{1}_{X(\mathbb{R})}$ in the image of $c_f$ cutting out $\mathscr{E}_{num}(Y_\mathbb{R})$. By definition of $\mathscr{E}_{num}(Y_\mathbb{R})$, the subalgebra $e\cdot \text{im}(c_f)\cdot e$ of endomorphisms of $f_*\mathbf{1}_{X(\mathbb{R})}$ is a local ring. By the surjectivity of localisation for Chow groups, restriction exhibits the subalgebra $e|_U\cdot \text{im}(c_{f|_U})\cdot e|_U$ of $\End((f|_U)_*\mathbf{1}_{f^{-1}(U)(\mathbb{R})})$ as the quotient of a local ring. It is therefore local, so the split inclusion $\mathscr{E}_{num}(U_\mathbb{R})\to \mathscr{E}_{num}(Y_\mathbb{R})|_{U(\mathbb{R})}$ is an isomorphism.
\end{proof}
\section{Examples and applications}\label{Sec:Examples and applications}

In this section we examine some classes of singular real varieties where we can more explicitly describe the sheaf $\mathscr{E}(Y_\mathbb{R})$. The case of suitable isolated singularities is given in Proposition \ref{Prop:real blowup minimal}, varieties admitting small resolutions addressed in Proposition \ref{Prop:Small resolutions}, and the more involved case of Schubert varieties is treated in Theorem \ref{Thm:Geo is parity on schubert}.

We begin this section with an explicit computation of $\mathscr{E}(Y_\mathbb{R})$ on a cone.

\begin{ex}\label{Ex:explicit computation}
Consider the cone $Y=V(x^2+y^2-z^2)$ in $\mathbb{A}^3$. This has a resolution $f:X=Bl_0(Y)\to Y$ given by blowing up the singular point, with exceptional fibre $\mathbb{P}^1$. On real points, $X(\mathbb{R})$ is the cylinder $S^1\times \mathbb{R}$, and this map is given by pinching the circle $S^1\times \{0\}$ down to the cone point on $Y(\mathbb{R})$.

In this case, we claim the pushforward $f_*\mathbf{1}_{X(\mathbb{R})}$ is indecomposable, giving an isomorphism: \[\mathscr{E}(Y_\mathbb{R})\cong f_*\mathbf{1}_{X(\mathbb{R})}.\] 
To see this indecomposability, it suffices to show that this sheaf has no skyscraper summands at the singular point, as any nondense summand must be supported at $0$. Any subsheaf supported on $\{0\}$ is the pushforward of a sheaf on the point, and the isomorphisms \[i^!i_*\cong \text{Id}_{\{0\}}\cong i^*i_*\] show that any such sheaf on $Y(\mathbb{R})$ will be nonzero under the map \[i^!f_*\mathbf{1}_{X(\mathbb{R})}\to i^*f_*\mathbf{1}_{X(\mathbb{R})}.\] This map of graded vector spaces is zero however, as the total space of the blowup is the total space of a trivial line bundle on $\mathbb{P}^1(\mathbb{R})$, so the Gysin morphism
\[H^{*-1}(\mathbb{P}^1(\mathbb{R}),\mathbb{F}_2)\to H^*(\mathbb{P}^1(\mathbb{R}),\mathbb{F}_2)\] is multiplication by the first Stiefel-Whitney class of this trivial line bundle (see Example \ref{ex:Manifolds}).
\end{ex}

The following example also serves to show that this construction is sensitive to the variety structure of $X$, rather than just as a topological space.

\begin{ex}\label{Ex:Two planes}
Let $Y=V_1\cup V_2$ be the union of two generic linear $2$-planes in $\mathbb{A}^4$. Taking the disjoint union of these planes gives a resolution, and the pushforward of the constant sheaf along this map is the direct sum of the constant sheaves on each plane. This pushforward is therefore the real geometric extension. In particular, the stalk of this sheaf at the singular point has trivial first cohomology. As the space $Y(\mathbb{R})$ is homeomorphic to that of the previous example, but is not isomorphic as a variety, we see that $\mathscr{E}(Y_\mathbb{R})$ depends on the variety structure.
\end{ex}

The proof used in Example \ref{Ex:explicit computation} may be generalised to the following topological application.

\begin{prop}\label{Prop:real blowup minimal}
    Let $Y$ have an isolated singularity at a real point $y$, such that the blowup $\pi:Bl_yY\to Y$ has smooth exceptional divisor $E$. Assume that $E(\mathbb{R})$ is nonempty, and that the real normal line bundle of $E(\mathbb{R})$ in the real points of this blowup is topologically trivial. Then we have an isomorphism: \[\mathscr{E}(Y_\mathbb{R})\cong \pi_*\mathbf{1}_{Bl_y Y(\mathbb{R})}\]
    As such, for any resolution $f:X\to Y$, the fibre $F_y=f^{-1}(y)$ contains the cohomology of $E(\mathbb{R})$ as a direct summand, giving:
    \[\dim H^i(E(\mathbb{R}),\mathbb{F}_2) \leq \dim H^i(F_y(\mathbb{R}),\mathbb{F}_2).\]
    Since $E$ is a divisor, this fibre has dimension $d-1$.
\end{prop}

\begin{proof}
As in the previous example, we need to check that $\mathscr{E}(Y_\mathbb{R})$ has no skyscraper summands at the singular point. Any such summand would imply the map \[i^!f_*\mathbf{1}_{X(\mathbb{R})}\to i^*f_*\mathbf{1}_{X(\mathbb{R})}\] is nonzero. By base change we may identify this with the map of graded vector spaces 
\[H^{*-1}(E(\mathbb{R}),\mathbb{F}_2)\to H^{*}(E(\mathbb{R}),\mathbb{F}_2)\]
given by multiplication by the first Stiefel-Whitney class of the normal bundle of $E(\mathbb{R})$ in this blowup. This map is thus zero by our triviality assumption, so no skyscraper summands may exist, and this pushforward is indecomposable.
\end{proof}

\begin{remark}
    In the situation of the above proposition, the triviality condition of this normal line bundle is equivalent to the topological link $\text{lk}_y(Y(\mathbb{R}))$ having the maximal number of connected components:
\[ |\pi_0(\text{lk}_y(Y(\mathbb{R})))|=2\cdot |\pi_0(E(\mathbb{R}))|\]
One can see this equivalence by considering a tubular neighbourhood $E(\mathbb{R})_\epsilon$ of $E(\mathbb{R})$ in this blowup, and counting the connected components of $E(\mathbb{R})_\epsilon\setminus E(\mathbb{R})$.
\end{remark}

This situation invites a comparison to the complex setting with rational coefficients. Let $Y$ be the affine cone on a smooth complex projective variety $X$ embedded via very ample line bundle $L$. When we analyse this pushforward along the blowup, the decomposition theorem implies that the indecomposable summand is the perverse sheaf $\mathbf{IC}(Y(\mathbb{C}),\mathbb{Q})$. The perversity restriction on stalks implies that multiplication by the first Chern class has cokernel supported in degree at most $d$. This follows from, and is only slightly weaker than, the hard Lefschetz property on the rational cohomology of $X(\mathbb{C})$. The above real example may be interpreted as a maximal failure of the hard Lefschetz theorem, since $w_1(L_\mathbb{R})$, the analogue of $c_1(L)$, is zero. This failure of hard Lefschetz then translates to the existence of a codimension one fibre in any resolution of $Y$.

For small resolutions, we may recognise the real geometric extension as the pushforward of the constant sheaf.

\begin{prop}\label{Prop:Small resolutions}
Let $Y$ be a $d$-dimensional real variety with a smooth real point, and let $f:X\to Y$ be a small resolution of $Y$. Then the real geometric extension is the pushforward of the constant sheaf along this resolution: \[\mathscr{E}(Y_\mathbb{R})\cong f_*\mathbf{1}_{X(\mathbb{R})}.\]
In particular, the cohomology of $\mathscr{E}(Y_\mathbb{R})$ recovers the cohomology of \emph{any} small resolution, if one exists:
\[\mathscr{E}^i(Y(\mathbb{R}),\mathbb{F}_2)\cong H^i(X(\mathbb{R}),\mathbb{F}_2)\]
\end{prop}

\begin{proof}
By the convolution isomorphism of Example \ref{Ex:conv iso}, we may identify the endomorphisms of $f_*\mathbf{1}_{X(\mathbb{R})}$ with the $d$th Borel-Moore homology of the fibre product $X(\mathbb{R})\times_{Y(\mathbb{R})}X(\mathbb{R})$. Consider now the (open) restriction map to $\Delta(Y^{sm}(\mathbb{R}))$:
 \[H^{\mathrm{BM}}_d(X(\mathbb{R})\times_{Y(\mathbb{R})}X(\mathbb{R}),\mathbb{F}_2)\to H^{\mathrm{BM}}_d(\Delta(Y^{sm}(\mathbb{R})),\mathbb{F}_2)\]
A fundamental class of the closure of $\Delta(Y^{sm}(\mathbb{R}))$ exhibits the surjectivity of this map by base change. On the other hand, our assumption of smallness of the map $f$ gives that the complement $Z(\mathbb{R})$ of $\Delta(Y^{sm}(\mathbb{R}))$ has cohomological dimension less than $d$. The long exact sequence in Borel-Moore homology then yields the injectivity of the restriction map \[0=H_d^{\mathrm{BM}}(Z(\mathbb{R}),\mathbb{F}_2)\to H^{\mathrm{BM}}_d(X(\mathbb{R})\times_{Y(\mathbb{R})}X(\mathbb{R}),\mathbb{F}_2)\to H^{\mathrm{BM}}_d(\Delta(Y^{sm}(\mathbb{R})),\mathbb{F}_2)\] Identifying this restriction with restriction of endomorphisms (see \cite[\S A]{HoneWilliamson2025}) we conclude that $f_*\mathbf{1}_{X(\mathbb{R})}$ is the minimal summand extending $\mathbf{1}_{Y^{sm}(\mathbb{R})}$ on the smooth locus, completing the proof.
\end{proof}
\subsection{Real Schubert varieties and parity sheaves}
We conclude this section with an analysis of real geometric extensions on Schubert varieties. In this case, it is more convenient to describe the whole additive, graded category $\mathrm{Geo}(\mathscr{F}(\mathbb{R}),\mathbb{F}_2)$ of real geometric extensions on the real flag variety. Our main theorem is an equivalence of this category with $\mathrm{Par}_{ev}(\mathscr{F}(\mathbb{C}),\mathbb{F}_2)$, the category of even (mod two) parity sheaves on the complex flag variety. This equivalence follows at once from the motivic presentation of the category of parity sheaves, and our real geometric extensions in this context may be interpreted as the real realisation of the mod two motives underlying parity sheaves.

Recall that for $G$ a real split reductive group with Borel $B$, the flag variety $\mathscr{F}:=G/B$ has a Bruhat stratification into strata $X^\circ(w)$ indexed by $w$ in $W$ the Weyl group of $G$. Each stratum is isomorphic to an affine space of dimension $|w|$, the length of $w$, and the closure $X(w):=\overline{X^\circ(w)}$ is the Schubert variety associated to $w\in W$. These varieties have a nice family of resolutions; the Bott-Samelson resolutions $X(\underline{w})\to X(w)$ associated to reduced expressions $\underline{w}$ for $w$. These resolutions are naturally $B$ equivariantly projective, and this equivariance endows their fibres with affine pavings, needed for the following lemma.

\begin{lem}\label{Lem:affine paving}
Consider two Bott-Samelson resolutions $X(\underline{w})\to X(w)$, $X(\underline{u})\to X(u)$, interpreted as maps to the flag variety $\mathscr{F}$. The homology of their fibre product is completely controlled by its Chow groups; the following real and complex cycle class maps are isomorphisms:
\[\mathrm{CH}_\bullet(X(\underline{w})\times_{\mathscr{F}}X(\underline{u}))\otimes \mathbb{F}_2 \xrightarrow{\sim} H_{2\bullet}^{\mathrm{BM}}(X(\underline{w})(\mathbb{C})\times_{\mathscr{F}(\mathbb{C})}X(\underline{u})(\mathbb{C}),\mathbb{F}_2)\]
\[\mathrm{CH}_\bullet(X(\underline{w})\times_{\mathscr{F}}X(\underline{u}))\otimes \mathbb{F}_2\xrightarrow{\sim} H_{\bullet}^{\mathrm{BM}}(X(\underline{w})(\mathbb{R})\times_{\mathscr{F}(\mathbb{R})}X(\underline{u})(\mathbb{R}),\mathbb{F}_2)\]
\end{lem}
\begin{proof}
It is well known (see \cite{fulton2012intersectiontheory} over $\mathbb{C}$, \cite{BorelHaefliger} over $\mathbb{R}$) that these isomorphisms hold for real varieties $Y$ admitting an affine paving, a decomposition $Y=\bigcup_{i=1}^n Z_i$ with each $Z_i$ closed, with $Z_i\setminus Z_{i-1}\cong \mathbb{A}^{\alpha_i}$ for some $\alpha_i\in \mathbb{N}$.

It suffices to show these fibre products admit such a paving. The flag variety $\mathscr{F}$ admits an affine paving refining its decomposition into Bruhat cells, so using $B$-equivariance, we only need to show that the fibres of $X(\underline{w})\to X(w)$ admit affine pavings. For $T$ a maximal torus in $B$, $B$-equivariance implies we only need to find affine pavings for the fibres over $u$ fixed by $T$ in $X(w)$. For such a $u$, we may find another torus $T'$ in $B$ such that all weights of $T'$ on the tangent space of $\mathscr{F}$ at $u$ are negative. Let us consider the Białynicki-Birula cell decompositions \cite{BBdecomposition} for the $T'$ action on the (smooth) source and target of the $T'$ equivariant map \[X(\underline{w})\to \mathscr{F}.\] Our negativity assumption shows that $u\in \mathscr{F}$ has its attracting cell of dimension zero, so the cell decomposition of $X(\underline{w})$ induces a cell decomposition of the fibre over $u$. The $B$-equivariant projectivity of the Bott-Samelson resolution then implies this decomposition into affine space cells may be ordered to yield an affine paving, giving the result.
\end{proof}

\begin{cor}
For a Schubert variety $X(w)$, the real geometric extension and its numerical motivic variant agree.
\end{cor}

\begin{proof}
For a Bott-Samelson resolution $X(\underline{w})\to X(w)$, the map from the Chow groups of $X(\underline{w})\times_{X(w)}X(\underline{w})$ to its homology is surjective by Lemma \ref{Lem:affine paving}. The idempotent endomorphism cutting out the real geometric extension is therefore in the image of the real cycle class map.
\end{proof}

\begin{thm}\label{Thm:Geo is parity on schubert}
If $X(w)$ is a Schubert variety, then the cohomology groups, stalks and costalks of $\mathscr{E}({X(w)}_\mathbb{R})$ are isomorphic to those of the corresponding parity sheaf, with degrees divided by two:
\begin{align*}
    \mathscr{E}^i(X(w)(\mathbb{R}),\mathbb{F}_2):=\mathbb{H}^i\mathscr{E}(X(w)_\mathbb{R})&\cong \mathbb{H}^{2i}\mathscr{E}(X(w)_\mathbb{C})\\
    H^i(i_u^*\mathscr{E}(X(w)_\mathbb{R}))&\cong H^{2i}(i_u^*\mathscr{E}(X(w)_\mathbb{C}))\\
    H^i(i_u^!\mathscr{E}(X(w)_\mathbb{R}))&\cong H^{2i}(i_u^!\mathscr{E}(X(w)_\mathbb{C}))
\end{align*}
Furthermore, let $\mathrm{Geo}(\mathscr{F}(\mathbb{R}),\mathbb{F}_2)$ denote the additive category of sums and shifts of the sheaves $\mathscr{E}(X(w)_\mathbb{R})$ on the real flag manifold, and $\mathrm{Par}_{ev}(\mathscr{F}(\mathbb{C}),\mathbb{F}_2)$ the corresponding category of even parity sheaves. Then the above isomorphisms follow from an equivalence of categories \[\mathrm{Geo}(\mathscr{F}(\mathbb{R}),\mathbb{F}_2)\cong \mathrm{Par}_{ev}(\mathscr{F}(\mathbb{C}),\mathbb{F}_2)\] which doubles degrees.
\end{thm}

\begin{proof}
We will describe these categories via cycles on Bott-Samelson resolutions, and the cycle class maps will induce the desired equivalence. For $U$ be an upwards-closed subset of $W$ under Bruhat order, we will write $\mathscr{F}_U$ for the associated open set in $\mathscr{F}$, and similarly $X(\underline{w})_U$ for the base change of $X(\underline{w})$ to $\mathscr{F}_U$. To each $U$ we associate three graded categories, each with objects Bott-Samelson resolutions mapping to $\mathscr{F}$:
\begin{enumerate}
    \item $C_{\mathrm{CH}}(U)$ has morphisms: \[\Hom^\bullet_{C_{\mathrm{CH}}(U)}(X(\underline{w}),X(\underline{u}))=\text{CH}_{\bullet+|u|}(X(\underline{w})_U\times_{\mathscr{F}_U}X(\underline{u})_U)\otimes \mathbb{F}_2\]
    \item 
    $C_{\mathbb{C}}(U)$ has morphisms: \[\Hom^\bullet_{C_{\mathbb{C}}(U)}(X(\underline{w}),X(\underline{u}))=H^{\textrm{BM}}_{2\bullet+2|u|}(X(\underline{w})(\mathbb{C})_U\times_{\mathscr{F}(\mathbb{C})_U}X(\underline{u})(\mathbb{C})_U,\mathbb{F}_2)\]
    \item $C_{\mathbb{R}}(U)$ has morphisms: \[\Hom^\bullet_{C_{\mathbb{R}}(U)}(X(\underline{w}),X(\underline{u}))=H^{\textrm{BM}}_{\bullet+|u|}(X(\underline{w})(\mathbb{R})_U\times_{\mathscr{F}(\mathbb{R})_U}X(\underline{u})(\mathbb{R})_U,\mathbb{F}_2)\]
\end{enumerate}

Using Lemma \ref{Lem:affine paving} and the compatible localisation sequences, for all $U$ the canonical comparison maps between these categories induce isomorphisms of hom-sets:
\[\begin{tikzcd}
    &C_{\mathrm{CH}}(U)\arrow[dl,"\sim"']\arrow[dr,"\sim"]&\\
    C_{\mathbb{C}}(U)&&C_{\mathbb{R}}(U)
\end{tikzcd}\]

These maps also preserve the composition law (see \cite{fulton2012intersectiontheory} and \cite{BorelHaefliger} for complex and real cases respectively), so induce canonical equivalences of categories. Letting this common category over $U$ be $C(U)$, we define the additive, graded category $\overline{C}(U)$ as the closure of $C(U)$ under direct sums, formal grading shifts and idempotent completion. In the complex realisation $\overline{C}(\mathscr{F})$ provides a presentation of $\mathrm{Par}_{ev}(\mathscr{F}(\mathbb{C}),\mathbb{F}_2)$, the even parity sheaves on $\mathscr{F}$. In particular it is Krull-Schmidt with indecomposable objects indexed (up to even shift) by elements of $W$. On the other hand, in the real context, this category $\overline{C}(\mathscr{F})$ contains $\mathrm{Geo}(\mathscr{F}(\mathbb{R}),\mathbb{F}_2)$, and the description of indecomposable objects in $\overline{C}(U)$ shows that this inclusion is an equivalence. In other words, the lower summands of the pushforward of the constant sheaf along a Bott-Samelson resolution in the real context are real geometric extensions.
This then induces the desired equivalence:\begin{align*}\mathrm{Geo}(\mathscr{F}(\mathbb{R}),\mathbb{F}_2)&\cong \mathrm{Par}_{ev}(\mathscr{F}(\mathbb{C}),\mathbb{F}_2)\\
\mathscr{E}(X(w)_\mathbb{R})[n]&\leftrightarrow \mathscr{E}(X(w)_\mathbb{C})[2n]
\end{align*}
In this category, the cohomology functor is corepresentable by $\mathbf{1}_{\mathscr{F}}$, the unique indecomposable object of dense support, which realises to the constant sheaf in the real and complex settings, implying the result on cohomology.
The stalk and costalk functors are not representable directly, but are still intrinsic to this abstract Hecke category $\overline{C}$. For an inclusion $U\subset V$ we have the natural restriction functor \[\overline{C}(V)\to \overline{C}(U)\] and this map is surjective on morphisms by the localisation sequence in Chow groups. 
To describe the stalk and costalk at $u$, consider the principal subsets $U_{\geq u}$ and $U_{>u}=U_{\geq u}\setminus \{u\}$ of $W$. To any object $\mathcal{F}$ in $\overline{C}(\mathscr{F})$, we have the following two functors:
\[\ker\bigg{(}\hom^*_{\overline{C}(U_{\geq u})}(\mathbf{1}_{U_{\geq u}},\mathcal{F}_{U_{\geq u}})\to \hom^*_{\overline{C}(U_{> u})}(\mathbf{1}_{U_{> u}},\mathcal{F}_{U_{> u}})\bigg{)}\]
\[\ker\bigg{(}\hom^*_{\overline{C}(U_{\geq u})}(\mathcal{F}_{U_{\geq u}},\mathbf{1}_{U_{\geq u}})\to \hom^*_{\overline{C}(U_{> u})}(\mathbf{1}_{U_{> u}},\mathcal{F}_{U_{> u}})\bigg{)}\]
In view of the localisation sequences for sheaves, we recognise this first functor as the stalk, and the second as taking the dual of the costalk, shifted by $|w_0|$ the dimension of $\mathscr{F}$. 
\end{proof}
In forthcoming work we will use this category of real geometric extensions to understand perverse sheaves on the real affine Grassmannian.
\section{Further questions}\label{Sec:Further questions}

We conclude with some natural questions regarding these real geometric extensions $\mathscr{E}(Y_\mathbb{R})$. These sheaves share some key properties with intersection cohomology and on Schubert varieties they may be interpreted as a real version of parity sheaves. In this complex case, both of these constructions have an intrinsic description, as a sheaf on $Y(\mathbb{C})$ with stalks and costalks vanishing in certain degrees.

\begin{question}
Is there an intrinsic characterisation of the sheaf $\mathscr{E}(Y_\mathbb{R})$? 
\end{question}

In contrast to intersection cohomology and parity sheaves, Example \ref{Ex:Two planes} indicates that an answer cannot be purely topological. An answer to this question would clarify exactly what topological information about resolutions the object $\mathscr{E}(Y_\mathbb{R})$ contains. 

In a different direction, Corollary \ref{Cor:fibre minimality} implies that if $f:X\to Y$ is a resolution such that $\mathscr{E}(Y_\mathbb{R})$ is equal to $f_*\mathbf{1}_{X(\mathbb{R})}$, this resolution must be cohomologically minimal in the sense that the (fibrewise) cohomology of any resolution contains the cohomology of the fibres of $f$. While small resolutions have this property, Proposition \ref{Prop:real blowup minimal} shows that this property may occur more frequently, leading to the natural question:

\begin{question}
For a given resolution $f:X\to Y$ when do we have the following isomorphism: \[\mathscr{E}(Y_\mathbb{R})\cong f_*\mathbf{1}_{X(\mathbb{R})}.\]
\end{question}

Finally, there is the question of how different $\mathscr{E}(Y_\mathbb{R})$ and its numerical motivic variant $\mathscr{E}_{num}(Y_\mathbb{R})$ really are.
\begin{question}
Does there exist a real variety $Y$ where $\mathscr{E}(Y_\mathbb{R})$ is not isomorphic to $\mathscr{E}_{num}(Y_\mathbb{R})$? Is $\mathscr{E}(Y_\mathbb{R})$ Zariski-local? 
\end{question}

\bibliographystyle{plain}
\bibliography{Realalgvar}

\Addresses

\end{document}